\documentclass[12pt]{article}
\usepackage{amsmath,amsthm,amssymb}
\usepackage[left=3cm,right=3cm, top=2.5cm,bottom=2.5cm,bindingoffset=0cm]{geometry}

\usepackage{amscd}

\usepackage{hyperref}

\usepackage{mathtools}

\usepackage[all,cmtip]{xy}

\usepackage{MnSymbol} 

\usepackage[table]{xcolor}

\usepackage{wrapfig}
\usepackage{graphicx}

\makeatletter
\def\th@plain{%
  \thm@notefont{}
  \itshape 
}
\def\th@definition{%
  \thm@notefont{}
  \normalfont 
}
\makeatother

\newtheorem{lemma}{Lemma}[section]
\newtheorem{proposition}[lemma]{Proposition}

\newtheorem{remark-definition}[lemma]{Remark-Definition}
\newtheorem{theorem}[lemma]{Theorem}
\newtheorem{corollary}[lemma]{Corollary}

\newtheorem{proposition-conjecture}[lemma]{Proposition-conjecture}

\theoremstyle{definition}
\newtheorem{example}[lemma]{Example}

\newtheorem{definition}[lemma]{Definition}
\newtheorem{remark}[lemma]{Remark}

\newcommand{\Ker}{\mathrm{Ker}\,}

\newcommand{\St}{\mathrm{St}}

\newcommand{\rk}{\mathrm{rk}\,}

\newcommand{\gl}{\mathrm{gl}\,}

\newcommand{\goth}{\mathfrak}

\newcommand{\codim}{\mathrm{codim}\,}

\newcommand{\Sing}{\mathsf{Sing}}

\newcommand{\nv}{n_{{v}}}

\newcommand{\nh}{n_{{h}}}

\newcommand{\dimO}{\dim \mathcal O_{\mathrm{reg}}}

\newcommand{\codimO}{\codim \mathcal O_{\mathrm{reg}}}

\newcommand{\dimSt}{\dim \mathrm{St}_{\mathrm{reg}}}

\ifdefined\C
 \renewcommand{\C}{\mathbb{C}}
\else
 \newcommand{\C}{\mathbb{C}}
\fi

\newcommand\CP{\mathbb{C}\mathbb{P}}

\ifdefined\N
 \renewcommand{\N}{\mathbb{N}}
\else
 \newcommand{\N}{\mathbb{N}}
\fi

\newcommand{\charp}{\chi}

\def\minus{\hbox{-}}   

\usepackage{multirow}  

\usepackage[normalem]{ulem} 

\usepackage{tikz}
\usepackage{tikz-cd}
\usetikzlibrary{calc}

\definecolor{block}{RGB}{0,162,232}

\def\blockaux#1(#2,#3)#4(#5,#6){%
  \draw[fill={#1}]
  let \p1=(#2,#3),
      \p2=(#5,#6),
      \p3=(#2+#5,#3+#6),
      \p4=(#2+#5/2,#3+#6/2)
  in
    (\p1) rectangle (\p3)
    (\p4) node {$#4$}
  ;%
}

\begin{document}

\title{New techniques for calculation of \\ Jordan-Kronecker invariants for Lie algebras \\ and Lie algebra representations}
\author{I.\,K.~Kozlov\thanks{No Affiliation, Moscow, Russia. E-mail: {\tt ikozlov90@gmail.com} }
}
\date{}

\maketitle

\begin{abstract} We introduce two novel techniques that simplify calculation of Jordan-Kronecker invariants for a Lie algebra $\mathfrak{g}$ and for a Lie algebra representation $\rho$. First, the stratification of 	matrix pencils under strict equivalence puts restrictions on the Jordan-Kronecker invariants. Second, we show that the Jordan--Kronecker invariants of a semi-direct sum $\mathfrak{g} \ltimes_{\rho} V$ are sometimes determined by the Jordan--Kronecker invariants of the dual Lie algebra representation $\rho^*$. \end{abstract}

\tableofcontents

\section{Introduction}
\label{intro}

Let $A, B, C, D \in \operatorname{Mat}_{m\times n}(\mathbb{C})$ be complex $m\times n$ matrices. Two matrix pencils $A + \lambda B$ and $C + \lambda D$ are \textbf{strictly equivalent}, if there exist $P \in \operatorname{GL}(m)$ and $Q \in \operatorname{GL}(n)$ such that \[ PAQ = C, \qquad P BQ = D.\] The canonical form under the strict equivalence is the well-known \textbf{Kronecker Canonical Form} (KCF) (see \cite[Ch. XII, \S 5]{Gantmaher88} and Theorem~\ref{T:JK_operator}). If we can associate with any mathematical object a matrix pencil $A - \lambda B$, then it is natural to study its KCF. 

Following that idea, the \textbf{Jordan--Kronecker invariants of a} (complex finite-dimensional) \textbf{Lie algebra} $\mathfrak{g}$  were introduced by A.\,V.~Bolsinov and P.~Zhang in \cite{BolsZhang}. They can be described as follows (for details see Section~\ref{SubS:JKLieDef}).   Let $c^k_{ij}$ be the structural constants of $\mathfrak{g}$ (in some basis). For any $x, a \in \mathfrak{g}^*$ we get two skew-symmetric matrices \begin{equation} \label{Eq:PairLie} \mathcal{A}_x = \left( \sum_k c^k_{ij} x_k\right), \qquad \mathcal{A}_a = \left( \sum_k c^k_{ij} a_k\right).\end{equation}  The Jordan--Kronecker (JK) invariants of $\mathfrak{g}$ describe the KCF\footnote{It is natural to consider skew-symmetric matrix pencils under congruence. A classical result, see e.g. \cite[Ch. XII, \S 6, Theorem 6]{Gantmaher88}, is that two skew-symmetric matrix pencils are strictly equivalent
if and only if they are congruent. The canonical form for skew-symmetric matrix pencils under congruence is given by the Jordan--Kronecker theorem \ref{T:Jordan-Kronecker_theorem}.} (or, more precisely, the bundle) of a generic pencil \[\mathcal{P}_{x,a} = \mathcal{A}_x - \lambda \mathcal{A}_a, \qquad x, a \in \mathfrak{g}^*.\]  A \textbf{bundle} $\mathcal{B}(L)$ of a matrix pencil $L$ is the union (usually infinite) of all orbits that have the same KCF up to the specific values of the eigenvalues. Here ``generic pencil'' means that the bundle $\mathcal{B}\left(\mathcal{P}_{x,a}\right)$ is the same for all pairs $(x, a)$ from a non-empty Zariski open subset $U \subset \mathfrak{g}^* \times \mathfrak{g}^*$. Apart from  \cite{BolsZhang} and references therein, the Jordan--Kronecker invariants for  Lie algebras were studied by K.\,S.~Vorushilov in \cite{Vor1}, \cite{Vor2}, \cite{Vor3}, \cite{Vor4} and by A.\,A.~Garazha in \cite{Gar1}, \cite{Gar2}, \cite{Gar3}. 
 
Later, \textbf{Jordan--Kronecker invariants of a Lie algebra representation} $\rho: \mathfrak{g} \to \operatorname{gl} (V)$  were introduced by  A.\,V.~Bolsinov, A.\,M~Izosimov and the author in \cite{BolsIzosKozl19}.  They can be described as follows (for details see Section~\ref{SubS:JKReprDef}).  To each point  $x\in V$, the representation $\rho$ assigns a linear operator \begin{equation} \label{Eq:LInOperRepr}  \begin{aligned} R_x: &\goth g\to V, \\ R_x(\xi) = & \rho(\xi) x \in V. \end{aligned} \end{equation} The Jordan--Kronecker invariants of $\rho$ describe the KCF (or, more precisely, the bundle) of a generic pencil \[L_{x,a} = R_x - \lambda R_a, \qquad x, a \in V.\] As for Lie algebras, here ``generic'' means the bundle $\mathcal{B}\left(L_{x,a}\right)$ is the same for all pairs $(x, a)$ from a non-empty Zariski open subset $U \subset V \times V$. The Jordan--Kronecker invariants of  Lie algebra representations were also studied by the author in \cite{Kozlov23}.

In general, calculation of Jordan--Kronecker invariants may be a non-trivial task. We introduce two simple techniques that may help to simplify calculation of Jordan-Kronecker invariants for a Lie algebra $\mathfrak{g}$ and for a Lie algebra representation $\rho$.  

\subsection{JK invariants and stratification of bundles}

Our first idea is: \begin{enumerate} 

\item \textit{For a Lie algebra $\mathfrak{g}$ we can study the KCF for any pencil $\mathcal{P}_{x',a'}$, not necessarily a generic one.  This gives us information about the Jordan--Kronecker invariants, since the generic pencil $\mathcal{P}_{x,a}$ is a ``more generic" bundle than $\mathcal{P}_{x',a'}$. More precisely,  \[ \mathcal{B}\left( \mathcal{P}_{x',a'} \right) \subset \overline{\mathcal{B}}\left( \mathcal{P}_{x,a} \right). \] Similarly, we can study any pencil $L_{x',a'}$ for a Lie algebra representation $\rho$}. 

\end{enumerate}

The bundles of matrix pencil have a natural  stratification (see e.g.  \cite{Edelman99}, \cite{Pervouchine04}, \cite{Teran23}  and the references therein). About the stratification for bundles of skew-symmetric matrix pencil see e.g. \cite{Dmytryshyn14}.  Formally speaking, our first idea is to use the following obvious statement.

\begin{theorem} \label{T:GenericBundle} Let $V$ be a finite-dimensional complex vector space and let \[ f: V \to \operatorname{Mat}_{n\times m}(\mathbb{C}), \qquad f: x \to A_x\]  be a continuous map. Assume that there exist a  non-empty Zariski open subset $U \subset V \times V$ such that for all pairs $(x,a) \in U$ the pencils $L_{x,a} = A_{x} - \lambda A_a$ have the same bundle $\mathcal{B}(L_{x,a})$. Then for any $(x', a') \in V \times V$ there exist a sequence $(x_n,a_n) \in U, n\in \mathbb{N}$ such that \[ \lim_{n\to \infty} (x_n, a_n) = (x', a').\] Therefore, for the corresponding bundles we have the inclusion: \[ \mathcal{B}\left( L_{x',a'} \right) \subset \overline{\mathcal{B}}\left( L_{x,a} \right), \qquad \forall (x,a) \in U. \]   \end{theorem}

We use the standard notation $ \overline{S}$ for the closure, in the standard topology, of a set $S$. We identify the set of matrix pencils $A - \lambda B$, where $A, B \in \operatorname{Mat}_{m\times n} \left(\mathbb{C}\right)$ with $\mathbb{C}^{2mn}$, and we consider the standard topology in this set. $ \overline{\mathcal{B}}\left( L \right)$ denotes the closure of the bundle of a matrix pencil $L$.

\subsubsection{Generic pencils}

How can Theorem~\ref{T:GenericBundle} be useful? It may be hard to detemine what pencils  $\mathcal{P}_{x,a}$ (or $L_{x,a}$) are generic. The structure of Lie algebra $\mathfrak{g}$ (or a Lie algebra representation $\rho$) may impose restrictions on possible matrix pencils  $\mathcal{P}_{x,a}$ (respectively, $L_{x,a}$). For instance, \[\operatorname{corank}  \mathcal{P}_{x,a} \geq \operatorname{ind} \mathfrak{g}, \qquad \forall x, a \in \mathfrak{g}^*\] (see \cite{BolsZhang}). If the corresponding bundle $\mathcal{B}\left(\mathcal{P}_{x,a} \right)$ has maximal possible dimension, then the pair $(x,a)$ is automatically generic. 

Generic matrix bundles with fixed rank were described in \cite{Dmytryshyn16} and generic skew-symmetric matrix bundles were described in \cite{Dmytryshyn18}. Thus, the next theorem follows from Theorem~\ref{T:GenericBundle} and \cite[Theorem 3.1]{Dmytryshyn18}. 

\begin{theorem} \label{T:GenKron} Let $\mathfrak{g}$ be a complex finite-dimensional Lie algebra such that $\operatorname{ind} \mathfrak{g} > 0$. Assume that there exists a pair $(x,a) \in \mathfrak{g}^* \times \mathfrak{g}^*$ such that 

\begin{enumerate}

\item the KCF of $\mathcal{P}_{x,a} = \mathcal{A}_x - \lambda \mathcal{A}_a$ consists of Kronecker $\left( 2k_j -1 \right) \times \left( 2k_j -1 \right)$ blocks, $j=1,\dots, q= \operatorname{ind} \mathfrak{g}$;

\item and the absolute difference between sizes of any two Kronecker indices is no more than $1$, i.e.  \[\left| k_i - k_j \right| \leq 1, \qquad 1 \leq i < j \leq \operatorname{ind} \mathfrak{g}.\]

\end{enumerate}

Then the pair $(x,a)$ is generic, and the Jordan--Kronecker invariants of $\mathfrak{g}$ are the Kronecker indices \[ k_1, \dots, k_q, \qquad q =\operatorname{ind} \mathfrak{g}.\] \end{theorem}

\begin{remark} Note that the number of Kronecker blocks in Theorem~\ref{T:GenKron} is equal to $\operatorname{ind}{\mathfrak{g}}$. If the Lie algebra $\mathfrak{g}$ is Frobenious, that is $\operatorname{ind} \mathfrak{g} = 0$, then the "most generic" JK invariants are obviously $1\times 1$ Jordan blocks with distinct eigenvalues. \end{remark}

We prove a similar statement for JK invariants of Lie algebra representations in Section~\ref{S:JKStat} (see Theorem~\ref{T:GenLieReprKron}).

\subsection{ JK invariants for semi-direct sum \texorpdfstring{$\mathfrak{g} \ltimes_{\rho} V$}{g x V} and dual representation \texorpdfstring{$\rho^*$}{rho}}

Our second idea is:

\begin{enumerate} \setcounter{enumi}{1}

\item \textit{The JK invariants of a semi-direct sum $\mathfrak{g} \ltimes_{\rho} V$ with a commutative ideal $V$ are sometimes determined by the JK invariants of the dual representation $\rho^*$.}

\end{enumerate} 

In particular we prove the following statement. We use the notations for the JK invariants from Section~\ref{S:BasicDef}.

\begin{theorem} \label{T:JKInvSumDualRepr} Assume that the JK invariants of the dual representation $\rho^*:\mathfrak{g} \to \operatorname{gl}(V^*)$ are 

\begin{enumerate}

\item Jordan tuples  $J_{\lambda_i} \left(n_{i1}, \dots, n_{i s_i}\right)$, $i=1,\dots, d$;

\item vertical indices $v_j(\rho^*)$, $j=1,\dots, q$.

\end{enumerate}

Then the JK invariants of the Lie algebra $\mathfrak{q} = \mathfrak{g} \ltimes_{\rho} V$ are 

\begin{enumerate}

\item Jordan tuples  $J^{\operatorname{skew}}_{\lambda'_i} \left(2m_{i1}, \dots, 2m_{i t_i}\right)$, $i=1,\dots, d$, where \[m_{i1} + \dots + m_{it_i} = n_{i1} + \dots + n_{is_i}, \qquad i = 1,\dots, d. \] 

\item Kronecker indices \[ k_j = v_j(\rho^*), \qquad j=1,\dots, q.\]

\end{enumerate}

\end{theorem}

 We prove this theorem in Section~\ref{SubS:ProofThDual}. See Section~\ref{S:ExDual} for a compelling illustration of this theorem's application. There we compare the Jordan-Kronecker invariants from Tables~\ref{table:1} and \ref{table:2}, corresponding to the sums of standard representation of classical matrix Lie algebras. We can easily see the following.
 
 \begin{itemize}
 
 \item The right columns are "the same" (all the Kronecker indices  coincide). 
 
 \item The indices in the middle columns are "so-so".
 
 \item And the indices in the left column are entirely different. 
 \end{itemize}

 This specific example served as the primary motivation for Theorem~\ref{T:JKInvSumDualRepr}.

\textbf{Conventions.} All Lie algebras and representations are finite-dimensional.  For simplicity sake we assume that underlying
field is $\mathbb{C}$, although all statements remain true for an arbitrary algebraically closed field $\mathbb{K}$  with $\operatorname{char} \mathbb{K} = 0$. We denote $\bar{\mathbb{C}} = \mathbb{C} \cup \left\{ \infty \right\}$. For short, we use JK as an abbreviation of Jordan-Kronecker and KCF as an abbreviation for Kronecker Canonical Form.

\par\medskip

\textbf{Acknowledgements.} The author would like to thank A.\,V.~Bolsinov and A.\,M.~Izosimov for useful comments.

\section{Basic definitions} \label{S:BasicDef}

In this section we recall the definition of the JK invariants and their basic properties. 

\subsection{Jordan-Kronecker invariants of representations}
\label{S:JK_Operator}

This section is mostly a retelling of \cite{BolsIzosKozl19}. We leave it for the sake of completeness.

\subsubsection{Kronecker Canonical Form of a pair of linear maps}
\label{SubS:JK_Operator}

In this section, we recall the Kronecker canonical form (KCF) for a pair of linear maps.  We first state the theorem in the matrix form, and then discuss the corresponding invariants.

\begin{theorem}[Kronecker Canonical Form \cite{Gantmaher88}]
\label{T:JK_operator} Any complex matrix pencil $A + \lambda B$ is strictly equivalent to one of the following matrix pencils:
\begin{equation}
\label{Eq:JK_Operator} P\left(A + \lambda B\right)Q  =
\left(\begin{matrix}
0_{m, n} &     &        &      \\
    & \!\!\! A_1 + \lambda B_1 &        &      \\
    &     & \!\!\!\! \ddots &      \\
    &     &        & \!\!\!  A_k + \lambda B_k  \\
\end{matrix}\right),
\end{equation} where $0_{m, n}$ is the zero $m \times n$-matrix, and each pair of the corresponding blocks   $A_i$ and $B_i$ takes one of the following forms:

1. Jordan block with eigenvalue $\lambda_0 \in \mathbb{C}$
 \[
 A_i =\left( \begin{matrix}
   \lambda_0 &1&        & \\
      & \lambda_0 & \ddots &     \\
      &        & \ddots & 1  \\
      &        &        & \lambda_0   \\
    \end{matrix} \right),
\quad  B_i=  \left( \begin{matrix}
    1 & &        & \\
      & 1 &  &     \\
      &        & \ddots &   \\
      &        &        & 1   \\
    \end{matrix} \right).
\]

2. Jordan block with eigenvalue $\infty$
\[
A_i = \left( \begin{matrix}
   1 & &        & \\
      &1 &  &     \\
      &        & \ddots &   \\
      &        &        & 1   \\
    \end{matrix}  \right),
\quad B_i = \left( \begin{matrix}
    0 & 1&        & \\
      & 0 & \ddots &     \\
      &        & \ddots & 1  \\
      &        &        & 0   \\
    \end{matrix}  \right).
 \]

  3. Horizontal Kronecker block 
\[
A_i= \left(
 \begin{matrix}
    0 & 1      &        &     \\
      & \ddots & \ddots &     \\
      &        &   0    & 1  \\
    \end{matrix}  \right),\quad
B_i = \left(\begin{matrix}
   1 & 0      &        &     \\
      & \ddots & \ddots &     \\
      &        & 1    &  0  \\
    \end{matrix}  \right)
\]

   4. Vertical Kronecker block
\[
A_i= \left(
  \begin{matrix}
  0  &        &    \\
  1   & \ddots &    \\
      & \ddots & 0 \\
      &        & 1  \\
  \end{matrix}
 \right), \quad
B_i = \left(
  \begin{matrix}
  1  &        &    \\
  0   & \ddots &    \\
      & \ddots & 1 \\
      &        & 0  \\
  \end{matrix}
 \right).
\]

 The number and types of blocks in decomposition  \eqref{Eq:JK_Operator} are uniquely defined up to permutation.
 \end{theorem}

\begin{remark} Unlike \cite{BolsIzosKozl19} (and similar to \cite{Gantmaher88}) there are no negative signs in the matrices $B_i$.  It doesn't affect much: there is a change of signs in some formulas (e.g. one should consider $\operatorname{Ker} (A- \lambda B)$ instead of $\operatorname{Ker} (A + \lambda B)$) but the JK invariants do not change. 
\end{remark}

It is convenient to regard the zero block $0_{m, n}$ in \eqref{Eq:JK_Operator} as a block-diagonal matrix that is composed of $m$ vertical Kronecker blocks of size  $1\times 0$ and $n$ horizontal Kronecker blocks of size $0\times 1$.

\begin{definition}
The \textit{horizontal indices} $\mathrm{h}_1, \dots, \mathrm{h}_p$ of the pencil $\mathcal P=\{A+\lambda B\}$ are defined to be the horizontal dimensions (widths) of horizontal Kronecker blocks: \[ A_i + \lambda B_i  = \underbrace{\left(
 \begin{matrix}
    \lambda & 1      &        &     \\
      & \ddots & \ddots &     \\
      &        &  \lambda    & 1  \\
    \end{matrix}  \right)}_{\mathrm{h}_i}.\] Similarly, the \textit{vertical indices} $\mathrm{v}_1, \dots, \mathrm{v}_q$  are the vertical dimensions (heights) of vertical blocks: \[ A_i + \lambda B_i  = \left. \left(
  \begin{matrix}
  \lambda  &        &    \\
  1   & \ddots &    \\
      & \ddots & \lambda \\
      &        & 1  \\
  \end{matrix}
 \right)\right\rbrace{\scriptstyle \mathrm{v}_i}.\] 
\end{definition}
In particular, in view of the above interpretation of the $0_{m,n}$ block, the first $m$ horizontal indices and first $n$ vertical indices are equal to $1$.

\subsubsection{Definition and notations for Jordan--Kronecker invariants} \label{SubS:JKReprDef}

Consider a finite-dimensional linear representation  $\rho: \goth g \to \gl(V)$ of a finite-dimensional Lie algebra $\goth g$.  To each point  $x\in V$, the representation $\rho$ assigns a linear operator  \begin{align*} R_x: &\goth g\to V, \\ R_x(\xi) = & \rho(\xi) x \in V. \end{align*}  Note that the matrix of $R_x$ is a $\operatorname{dim} V \times \dim \mathfrak{g}$ matrix.

\begin{definition} \label{Def:AlgTypePenc} The \textbf{algebraic type} of a pencil $L_{a,b} = R_a + \lambda R_b$, $a, b \in V$ is the following collection of discrete invariants:

\begin{itemize}

\item the sizes of Jordan blocks grouped by eigenvalues,

\item horizontal and vertical Kronecker indices.

\end{itemize}
\end{definition}

In Definiton~\ref{Def:AlgTypePenc} we distinguish horizontal and vertical indices. Note, that the algebraic type of $L_{a,b}$ is just a way to encode the bundle $\mathcal{B}\left(L_{a,b}\right)$.

\begin{definition}
\textbf{The Jordan--Kronecker invariant} of $\rho$ is the algebraic type of a pencil  $L_{a,b} = R_a + \lambda R_b$ for a generic pair $a, b \in V$.
\end{definition}

We denote the JK invariants of a Lie algebra representation $\rho$ as \begin{equation} \label{Eq:NotationJKRepr} J_{\lambda_1}(n_{11}, \dots, n_{1s_1}),   \dots,  J_{\lambda_d}(n_{d1}, \dots, n_{ds_d}), \quad \mathrm{h}_1(\rho),\dots,\mathrm{h}_p(\rho), \quad \mathrm{v}_1(\rho), \dots, \mathrm{v}_q(\rho). \end{equation}

Here  $\mathrm{h}_i(\rho)$ are the horizontal indices, and $\mathrm{v}_j(\rho)$  are the vertical indices of a generic pencil. We also call them \textbf{horizontal and vertical indices} of the representation $\rho$.   Each \textbf{Jordan tuple} $J_{\lambda_i}(n_{i1}, \dots, n_{is_i})$ represent Jordan blocks with sizes  $n_{i1} \times n_{i1}, \dots, n_{is_i} \times n_{is_i}$ and the same eigenvalue $\lambda_i$. In \eqref{Eq:NotationJKRepr} we regard $\lambda_i$ as (distinct) formal variables, since the eigenvalues $\lambda_i$ depend on $(x,a)$. We always assume that the numbers in the Jordan tuples $J_{\lambda_i}(n_{i1}, \dots, n_{is_i})$ are sorted in the descending order: \[n_{i1} \geq n_{i2} \geq \dots \geq n_{is_i}.\]

\subsection{Properties of Jordan-Kronecker invariants} \label{S:InterpretJK}

In this section we describe some properties of the Jordan-Kronecker invariants that we need below. 

\subsubsection{Kronecker blocks} 

In general, Kronecker blocks for a pencil $A +\lambda B$ correspond to the polynomial solutions of the following equations~\eqref{onceagain} and \eqref{onceagaindual}. For details see \cite{BolsIzosKozl19}.

\begin{proposition}[\cite{BolsIzosKozl19}]
\label{L:ChainsRestrCor} Consider a matrix pencil $A+ \lambda B$.
Horizontal Kronecker indices $\mathrm{h}_1, \dots, \mathrm{h}_p$ are given by $\mathrm{h}_i = r_i + 1$, where $r_1, \dots, r_p$ are the {minimal} degrees of independent solutions  $u(\lambda)=\sum_{j=0}^{l} u_j
\lambda^j$ of the equation
\begin{equation}
\label{onceagain}
(A+ \lambda B) u(\lambda)=0.
\end{equation} Similarly, vertical Kronecker indices $\mathrm{v}_1, \dots, \mathrm{v}_q$ are given by $\mathrm{v}_i = r'_i + 1$, where $r'_1, \dots, r'_q$ are the {minimal} degrees of independent solutions $u(\lambda)=\sum_{j=0}^{l} u_j
\lambda^j$  of the dual problem
\begin{equation}
\label{onceagaindual}
(A+ \lambda B)^* u(\lambda)=0.
\end{equation}
\end{proposition}

\subsubsection{Jordan blocks} 

The Jordan blocks correspond to \textit{elementary divisors} (see \cite{Gantmaher88}), which are defined using minors of a pencil $A+\lambda B$. The \textbf{rank of a pencil} $\mathcal P=\{A+\lambda B\}$ is the number \[\rk \mathcal P = \max_{\lambda\in \mathbb C} \rk (A+\lambda B).\] A value $\lambda_0 \in \mathbb{C} \cup \left\{ \infty\right\}$ is \textbf{regular} if $\operatorname{rk} (A+ \lambda_0 B) = \operatorname{rk}\mathcal{P}$. \begin{definition}
The \textbf{characteristic polynomial}  $\charp(\alpha, \beta)$ of the pencil $\mathcal P$ is defined as the greatest common divisor of all  the $r\times r$ minors of the matrix  $\alpha A + \beta B$,  where $r=\rk \mathcal P$.
\end{definition}

We can also regard the  homogeneous polynomial $\charp(\alpha, \beta)$  as  a polynomial in  $\lambda = \frac{\beta}{\alpha}$. We will use the following simple fact about Jordan blocks.

\begin{proposition}
\label{cor2} For any matrix pencil $A + \lambda B$:

\begin{enumerate}

\item There are Jordan blocks with eigenvalue $\lambda_0 \in \overline{\mathbb{C}}$ if and only if the rank of  $A-\lambda_0 B$ drops,  i.e.  $\rk (A-\lambda_0 B) < r=\rk \mathcal P$. The infinite eigenvalue appears in the case when $\rk B <r$.

\item Let $J_{\lambda_i} \left(n_{i1}, \dots, n_{is_i}\right)$ be the Jordan tuple corresponding an eigenvalue $\lambda_i \in \overline{\mathbb{C}}$. Then the multiplicity $n_{i1} + \dots + n_{i s_i}$ of the eigenvalue $\lambda_i$ coincides with the multiplicity of the corresponding root  $-\lambda_i = (\beta_i : \alpha_i)$ of the characteristic equation  $\charp(\alpha, \beta)=0$.

\end{enumerate}

\end{proposition}

\subsection{Jordan--Kronecker invariants of Lie algebras} \label{SubS:JKLieDef}

We briefly describe necessary facts about JK invariants for Lie algebras (for details, see \cite{BolsZhang}).

\subsubsection{Jordan--Kronecker theorem} 

First, let us recall the canonical form for a pair of skew-symmetric forms. This theorem, which we call the Jordan--Kronecker theorem, is a classical result that goes back to Weierstrass and Kronecker.  A proof of it can be found in \cite{Thompson}, which is based on \cite{Gantmaher88}.

\begin{theorem}[Jordan--Kronecker theorem]\label{T:Jordan-Kronecker_theorem}
Let $A$ and $B$ be skew-symmetric bilinear forms on a
finite-dimension vector space $V$ over a field $\mathbb{K}$ with $\textmd{char }  \mathbb{K} =0$. If the field $\mathbb{K}$
is algebraically closed, then there exists a basis of the space $V$
such that the matrices of both forms $A$ and $B$ are block-diagonal
matrices:

{\footnotesize
$$
A =
\begin{pmatrix}
A_1 &     &        &      \\
    & A_2 &        &      \\
    &     & \ddots &      \\
    &     &        & A_k  \\
\end{pmatrix}
\quad  B=
\begin{pmatrix}
B_1 &     &        &      \\
    & B_2 &        &      \\
    &     & \ddots &      \\
    &     &        & B_k  \\
\end{pmatrix}
$$
}

where each pair of corresponding blocks $A_i$ and $B_i$ is one of
the following:

\begin{itemize}

\item Jordan block with eigenvalue $\lambda_i \in \mathbb{K}$: {\scriptsize  \begin{equation} \label{Eq:JordBlockL} A_i =\left(
\begin{array}{c|c}
  0 & \begin{matrix}
   \lambda_i &1&        & \\
      & \lambda_i & \ddots &     \\
      &        & \ddots & 1  \\
      &        &        & \lambda_i   \\
    \end{matrix} \\
  \hline
  \begin{matrix}
  \minus\lambda_i  &        &   & \\
  \minus1   & \minus\lambda_i &     &\\
      & \ddots & \ddots &  \\
      &        & \minus1   & \minus\lambda_i \\
  \end{matrix} & 0
 \end{array}
 \right)
\quad  B_i= \left(
\begin{array}{c|c}
  0 & \begin{matrix}
    1 & &        & \\
      & 1 &  &     \\
      &        & \ddots &   \\
      &        &        & 1   \\
    \end{matrix} \\
  \hline
  \begin{matrix}
  \minus1  &        &   & \\
     & \minus1 &     &\\
      &  & \ddots &  \\
      &        &    & \minus1 \\
  \end{matrix} & 0
 \end{array}
 \right)
\end{equation}} \item Jordan block with eigenvalue $\infty$ {\scriptsize \begin{equation} \label{Eq:JordBlockInf}
A_i = \left(
\begin{array}{c|c}
  0 & \begin{matrix}
   1 & &        & \\
      &1 &  &     \\
      &        & \ddots &   \\
      &        &        & 1   \\
    \end{matrix} \\
  \hline
  \begin{matrix}
  \minus1  &        &   & \\
     & \minus1 &     &\\
      &  & \ddots &  \\
      &        &    & \minus1 \\
  \end{matrix} & 0
 \end{array}
 \right)
\quad B_i = \left(
\begin{array}{c|c}
  0 & \begin{matrix}
    0 & 1&        & \\
      & 0 & \ddots &     \\
      &        & \ddots & 1  \\
      &        &        & 0   \\
    \end{matrix} \\
  \hline
  \begin{matrix}
  0  &        &   & \\
  \minus1   & 0 &     &\\
      & \ddots & \ddots &  \\
      &        & \minus1   & 0 \\
  \end{matrix} & 0
 \end{array}
 \right)
 \end{equation} } \item   Kronecker block {\scriptsize \begin{equation} \label{Eq:KronBlock} A_i = \left(
\begin{array}{c|c}
  0 & \begin{matrix}
   1 & 0      &        &     \\
      & \ddots & \ddots &     \\
      &        & 1    &  0  \\
    \end{matrix} \\
  \hline
  \begin{matrix}
  \minus1  &        &    \\
  0   & \ddots &    \\
      & \ddots & \minus1 \\
      &        & 0  \\
  \end{matrix} & 0
 \end{array}
 \right) \quad  B_i= \left(
\begin{array}{c|c}
  0 & \begin{matrix}
    0 & 1      &        &     \\
      & \ddots & \ddots &     \\
      &        &   0    & 1  \\
    \end{matrix} \\
  \hline
  \begin{matrix}
  0  &        &    \\
  \minus1   & \ddots &    \\
      & \ddots & 0 \\
      &        & \minus1  \\
  \end{matrix} & 0
 \end{array}
 \right)
 \end{equation} }
 \end{itemize}

\end{theorem}

\begin{remark} Each Kronecker block is a $(2k_i-1) \times (2k_i-1)$ block, where
$k_i \in \mathbb{N}$. If $k_i=1$, then the blocks are $1\times 1$
zero matrices
\[
A_i =
\begin{pmatrix}
0
\end{pmatrix} \quad  B_i=
\begin{pmatrix}
0
\end{pmatrix}
\]
\end{remark}

\begin{remark} \label{Rem:JordSign} Note that in \cite{BolsZhang} the matrix $B_i$ for Jordan blocks is taken with the opposite sign. This is a question of convention --- whether $\lambda_i$ or $-\lambda_i$ should be called an eigenvalue. \end{remark}

We call a decomposition of $V$ into a sum of subspaces corresponding to Jordan and Kronecker blocks a \textbf{Jordan-Kronecker decomposition}:  \begin{equation} \label{Eq:JKDecomp}  V = \left( \bigoplus_{j=1}^N  V_{J_{\lambda_j, 2n_j}} \right) \oplus  \left( \bigoplus_{i=1}^q V_{K_i}\right). \end{equation} 

\subsubsection{Notations for Jordan--Kronecker invariants}

Jordan--Kronecker invariants of finite-dimensional Lie algebras were introduced in \cite{BolsZhang}. They arise as a combination of two simple facts:

\begin{itemize}

\item Let  $\mathfrak{g}$ be a finite-dimensional Lie algebra and  $\mathfrak{g}^*$ be its dual space. Any element $x\in \mathfrak{g}^*$ defines a skew-symmetric bilinear form $\mathcal{A}_x = \left( c_{ij}^k x_k\right)$, where $c_{ij}^k$ are the structural constants of $\mathfrak{g}$.  

\item There is a canonical form for a pencil of skew-symmetric bilinear forms given by the Jordan--Kronecker theorem (Theorem~\ref{T:Jordan-Kronecker_theorem}, see, e.g., \cite{Thompson, Gantmaher88}). 

\end{itemize}

The Jordan--Kronecker invariants of $\mathfrak{g}$ are determined by the KCF of the pencil $\mathcal A_{x+ \lambda a}$, given by \eqref{Eq:PairLie}, for a generic pair $(x,a)\in \mathfrak{g}^*\times \mathfrak{g}^*$. In short, there are two kinds of blocks in the Jordan--Kronecker theorem: Kronecker blocks and Jordan blocks. 

\begin{definition} The Jordan--Kronecker invariants of $\mathfrak{g}$ are 

\begin{itemize}

\item sizes of Kronecker blocks,

\item sizes of Jordan blocks grouped by eigenvalues

\end{itemize}
for a generic pencil $\mathcal{A}_{x + \lambda a}$. 
\end{definition}

The JK invariants of $\mathfrak{g}$ are a collection of Jordan tuples and Kronecker sizes: \begin{equation} \label{Eq:NotationJK} J^{\mathrm{skew}}_{\lambda_1}(2n_{11}, \dots, 2n_{1s_1}), \quad  \dots, \quad  J^{\mathrm{skew}}_{\lambda_p}(2n_{p1}, \dots, 2n_{ps_p}), \quad 2k_1 -1, \quad \dots,  \quad 2k_q - 1. \end{equation} Each \textbf{Kronecker size} $2k_j-1$ represent a  Kronecker $(2k_j-1)\times(2k_j-1)$ block  (the number $k_j$ is its \textbf{Kronecker index}).. Each \textbf{Jordan tuple} $J^{\mathrm{skew}}_{\lambda_i}(2n_{i1}, \dots, 2n_{is_i})$ represent Jordan blocks with sizes  $2n_{i1} \times 2n_{i1}, \dots, 2n_{is_i} \times 2n_{is_i}$ and the same eigenvalue $\lambda_i$. We always assume that the numbers in the Jordan tuples $J^{\mathrm{skew}}_{\lambda_i}(2n_{i1}, \dots, 2n_{is_i})$ are sorted in the descending order: \[n_{i1} \geq n_{i2} \geq \dots \geq n_{is_i}.\]  In \eqref{Eq:NotationJK} we regard $\lambda_i$ as (distinct) formal variables, since the eigenvalues $\lambda_i$ depend on $(x,a)$. Nevertheless, the number and sizes of blocks are the same for a generic pair $(x,a)\in \mathfrak{g}^*\times \mathfrak{g}^*$.

\subsubsection{Core and mantle subspaces}

In the proof of Lemma~\ref{L:CalcJKCoreMantle} we need the following definitions. 

\begin{definition} Consider a pencil of skew-symmetric forms $\left\{ A_{\lambda} = A + \lambda B\right\}$.

\begin{enumerate}

\item The \textbf{core} subspace is \[ K = \sum_{\lambda - regular} \operatorname{Ker} A_{\lambda}. \] 

\item The \textbf{mantle} subspace is the skew-orthogonal complement to the core (w.r.t. any regular form $A_{\mu}$) \[ M = K^{\perp}. \] 

\end{enumerate}

\end{definition} 

There is an easy description of the core and mantle.

\begin{proposition} \label{Cor:CoreMantle} For any JK decomposition~\eqref{Eq:JKDecomp}. 
\begin{enumerate}

\item The core subspace $K$ of $V$ is the direct sum of core subspaces of Kronecker subspaces $V_{K_i}$: \[ K = \bigoplus_{i=1}^q \left(K \cap V_{K_i}\right).\] If $e_1, \dots, e_{k_i-1},  f_1, \dots, f_{k_i}$ is a standard basis of $V_{K_i}$, then the core subspace of $V_{K_i}$ is \[K \cap V_{K_i} = \operatorname{span} \left(f_1, \dots, f_{k_i} \right).\]

\item The mantle subspace $M$ is the core plus all Jordan blocks: \[ M = K \oplus V_J. \]

\end{enumerate}

\end{proposition} 

Simply speaking, the core subspace $K$ is spanned by the basis vectors corresponding to the down-right zero matrices of Kronecker blocks, like this one:  \[ A_i + \lambda B_i = \left(
\begin{array}{c|c}
  0 & \begin{matrix}
   1 & \lambda      &        &     \\
      & \ddots & \ddots &     \\
      &        & 1    & \lambda  \\
    \end{matrix} \\
  \hline
  \begin{matrix}
  \minus1  &        &    \\
  \minus \lambda   & \ddots &    \\
      & \ddots & \minus1 \\
      &        & \minus\lambda  \\
  \end{matrix} &\cellcolor{blue!25} 0 
 \end{array}
 \right).
 \]

\section{Jordan--Kronecker invariants and stratification of pencils} \label{S:JKStat}

In this section we show how stratification of matrix pencils under strict equivalence may help us determine JK invariants.

\begin{itemize}

\item First, we describe stratification of pencils (Theorem~\ref{T:KCFStarts}) and  bundles (Theorem~\ref{T:BundleStrat}).

\item Then, we describe generic KCF with fixed rank (Theorem~\ref{T:GenKCFPenc}). After that in Theorem~\ref{T:GenLieReprKron} we state an analogue of Theorem~\ref{T:GenKron} for JK invariants of representations.   

\end{itemize}

We emphasize that, when calculating JK invariants, stratification of pencils can be used not only for generic pencils with fixed rank, but for pencils with any other restrictions on the KCF as well (see Remark~\ref{R:OtherRest}).

\subsection{Stratification of KCF} \label{Subs:StratPenc}

First, we introduce notations for some matrix  pencils. For each $k \in \mathbb{N}$ define $k\times k$ matrices \[ J_k(\mu) = \left( \begin{matrix} \mu & 1 & & \\ & \mu & \ddots & \\ & & \ddots & 1 \\ & & & \mu \end{matrix} \right), \qquad  I_k(\mu) =  \left( \begin{matrix} 1 &  & & \\ &  \ddots & & \\ & & \ddots &  \\ & & & 1 \end{matrix} \right), \] and for each $k=0,1,\dots, $ define the $k \times (k+1)$  matrices \[F_k  = \left( \begin{matrix}
    0 & 1      &        &     \\
      & \ddots & \ddots &     \\
      &        &   0    & 1  \\
    \end{matrix}  \right),\qquad G_k = \left(\begin{matrix}
   1 & 0      &        &     \\
      & \ddots & \ddots &     \\
      &        & 1    &  0  \\
    \end{matrix}  \right).\] Next, we denote the matrix pencils for the Jordan blocks with eigenvalue $\mu$ as \[ \mathcal{E}_k(\mu) = J_k(\mu) + \lambda I_k,\qquad \mathcal{E}_k(\infty) = \lambda J_k(0) + I_k.\] The matrix pencils for the horizontal Kronecker blocks we denote as \[ \mathcal{L}_k = \lambda G_k + F_k \] Then the matrix pencils for the vertical Kronecker blocks are $\mathcal{L}^T_k$.  Now let us describe the stratification of matrix bundles.

\begin{theorem}[\cite{Boley98}] \label{T:KCFStarts} Let $\mathcal{P}_1$ and $\mathcal{P}_2$ be two matrix pencils in KCF. Then,
$\overline{O^e}\left(\mathcal{P}_1 \right) \supset O^e\left(\mathcal{P}_2 \right)$ if and only if $\mathcal{P}_1$ can be obtained from $\mathcal{P}_2$ changing canonical blocks of $\mathcal{P}_2$ by applying a sequence of rules and each rule is one of the six types below:

\begin{enumerate}

\item \label{KFC_Rule1} $\mathcal{L}_{j-1} \oplus \mathcal{L}_{k+1} \rightsquigarrow \mathcal{L}_{j} \oplus \mathcal{L}_{k}$, $1 \leq j \leq k$;

\item  \label{KFC_Rule2} $\mathcal{L}^T_{j-1} \oplus \mathcal{L}^T_{k+1} \rightsquigarrow \mathcal{L}^T_{j} \oplus \mathcal{L}^T_{k}$, $1 \leq j \leq k$;

\item \label{KFC_Rule3} $\mathcal{L}_{j} \oplus \mathcal{E}_{k+1}(\mu) \rightsquigarrow \mathcal{L}_{j+1} \oplus  \mathcal{E}_{k}(\mu)$, $j,k =0,1,2, \dots$ and $\mu \in \overline{\mathbb{C}}$;

\item \label{KFC_Rule4}  $\mathcal{L}^T_{j} \oplus \mathcal{E}_{k+1}(\mu) \rightsquigarrow \mathcal{L}^T_{j+1} \oplus  \mathcal{E}_{k}(\mu)$, $j,k =0,1,2, \dots$ and $\mu \in \overline{\mathbb{C}}$;

\item  \label{KFC_Rule5} $\mathcal{E}_{j}(\mu)\oplus \mathcal{E}_{k}(\mu) \rightsquigarrow \mathcal{E}_{j-1}(\mu)\oplus \mathcal{E}_{k+1}(\mu) $, $1 \leq j \leq k$ and $\mu \in \overline{\mathbb{C}}$;

\item \label{KFC_Rule6} $\mathcal{L}_{p}\oplus \mathcal{L}^T_{q} \rightsquigarrow \bigoplus_{i=1}^t \mathcal{E}_{k_i}(\mu_i)$, if $p+q+1 = \sum_{i=1}^t k_i$ and $\mu_i \not = \mu_{i'}$ for  $i \not = i'$, $\mu_i \in \overline{\mathbb{C}}$.

\end{enumerate}

Observe that in the rules above any block $\mathcal{E}_0(\mu)$  should be understood as the empty matrix.

\end{theorem}

\begin{remark}
Note that the last transformation in Theorem~\ref{T:KCFStarts} (where a horizontal and vertical Kronecker blocks are replaced with Jordan blocks) increases the rank of the pencil by $1$, whereas other transformations in it preserve the rank. \end{remark}
 
 \begin{remark}  A reformulation of Theorem~\ref{T:KCFStarts} in terms of the Weyr characteristic is given in \cite[Theorem 7]{Teran23}  (see also the references therein).  \end{remark}

\subsubsection{Stratification for bundles of matrix pencils} \label{Subs:StratBund}

In this paper we mostly work with bundles $\mathcal{B}(L)$ of a pencil $L$, not with 
orbits $\mathcal{O}(L)$. The stratification of bundles is fairly straightforward, we briefly describe the result  from  \cite{Teran23}.

We denote by $\Phi$ the set of all one-to-one mappings of $\overline{\mathbb{C}}$ to itself. We also use the notation $\Psi$ for the set of all mappings from $\overline{\mathbb{C}}$ to itself (not necessarily one-to-one). Then, for $\psi \in \Psi$ we denote by $\psi(L)$ any pencil which is strictly equivalent to the pencil obtained from $\operatorname{KCF}(L)$ after replacing the Jordan blocks associated with the eigenvalue $\mu \in \overline{\mathbb{C}}$ by Jordan blocks of the same size associated with the eigenvalue $\psi(\mu) \in \overline{\mathbb{C}}$. 

For a given matrix pencil $L(\lambda) = A + \lambda B$, with $A,B \in \mathbb{C}^{m\times n}$, we denote:
\begin{gather*} \mathcal{O}(L) := \left\{  P\left(A  + \lambda B\right)Q : P \in \operatorname{GL}(m), \quad Q \in \operatorname{GL}(n) \right\}, \qquad (\text{orbit of } L(\lambda)), \\
\mathcal{B}(L) := \bigcup_{\varphi \in \Phi} \mathcal{O} \left(\varphi(L) \right) \qquad (\text{bundle of } L(\lambda)). \end{gather*} Note that $\mathcal{O} \left(\varphi(L) \right)$ is well defined, regardless of the particular pencil $\varphi(L)$, since, as mentioned above, all pencils  $\varphi(L)$ are strictly equivalent.

 \begin{theorem}[{\cite[Theorem 12]{Teran23}}] \label{T:BundleStrat} Let $L(\lambda)$ and $M(\lambda)$  be two matrix pencils of the same size. Then $\overline{\mathcal{B}} \left( M\right) \subseteq \overline{\mathcal{B}} \left( L\right)$  if
and only if $M(\lambda) \in \overline{\mathcal{O}}\left(\psi_c(L) \right)$, for some map $\psi: \overline{\mathbb{C}} \to  \overline{\mathbb{C}}$.  \end{theorem}

\subsubsection{Generic KFC with fixed rank} \label{Subs:GenKCF}

Using Theorem~\ref{T:KCFStarts}, it is not hard to find top-dimensional orbits of pencils with a given rank. Denote by $\mathcal{P}^{m\times n}_r$ of $m \times n$ matrix pencils with 
complex coefficients and rank at most $r$.
 
\begin{theorem}[{\cite[Theorem~3.5]{Teran08}}] \label{T:GenKCFPenc} Let $r$ be an integer with $1 \leq r \leq \min(m, n)$ if $m\not = n$ and $1 \leq r \leq n-1$ if $m = n$. Then the set $\mathcal{P}^{m\times n}_r$ is a closed set which has exactly $r+1$ irreducible components in the Zariski topology. These irreducible components are $\overline{\mathcal{O}\left( \mathcal{K}_a\right)}$, for $a =0,1,\dots, r$, where \[ \mathcal{K}_a(\lambda) :=\operatorname{diag} \left( \underbrace{\mathcal{L}_{\alpha+1},\dots, \mathcal{L}_{\alpha+1}}_{s},  \underbrace{\mathcal{L}_{\alpha},\dots, \mathcal{L}_{\alpha}}_{n-r-s}, \underbrace{\mathcal{L}^T_{\beta+1},\dots, \mathcal{L}^T_{\beta+1}}_{t},  \underbrace{\mathcal{L}^T_{\beta},\dots, \mathcal{L}^T_{\beta}}_{m-r-t}\right)\] with \[a = \alpha(n-r) + s, \qquad \text{ and } \qquad r-a = \beta(m-r)  + t\]  being the Euclidean divisions of $a$ and $r-a$ by, respectively, $n-r$ and $m-r$. \end{theorem}

 \subsection{Generic Jordan--Kronecker invariants of representations}

 Theorem~\ref{T:GenKron} is an analogue of the following theorem for Lie algebras and can be proved analoguesly. 

\begin{theorem} \label{T:GenLieReprKron} Let  $\rho: \goth g \to \gl(V)$ be a Lie algera representation, \[ n = \dim \mathfrak{g}, \qquad m = \dim V.\] Denote \begin{equation} \label{Eq:RankCond} r = \dim \mathfrak{g} -  \dimSt = \dim V - \codimO.\end{equation} Assume that any of the following 2 conditions hold: \[m \not = n \qquad  \text{ or } \qquad r < \min(m, n).\] Also assume that there exists a pair $(x,a) \in \mathfrak{g}^* \times \mathfrak{g}^*$ such that the KCF of $L_{x,a} = R_x - \lambda R_a$ has the form $\mathcal{K}_d(\lambda)$ from Theorem~\ref{T:GenKCFPenc} for some $d=0,1\dots, r$. In other words, \[ \mathcal{K}_d(\lambda) :=\operatorname{diag} \left( \underbrace{\mathcal{L}_{\alpha+1},\dots, \mathcal{L}_{\alpha+1}}_{s},  \underbrace{\mathcal{L}_{\alpha},\dots, \mathcal{L}_{\alpha}}_{n-r-s}, \underbrace{\mathcal{L}^T_{\beta+1},\dots, \mathcal{L}^T_{\beta+1}}_{t},  \underbrace{\mathcal{L}^T_{\beta},\dots, \mathcal{L}^T_{\beta}}_{m-r-t}\right),\]   
where \[d = \alpha(n-r) + s, \qquad \text{ and } \qquad r-d = \beta(m-r)  + t\]  are the Euclidean divisions. Then the pair $(x,a)$ is generic, and the Jordan--Kronecker invariants of $\mathfrak{g}$ are the following:

\begin{itemize}

\item The horizontal indices \[\underbrace{\alpha+2,\dots, \alpha+2}_{s}, \qquad \underbrace{\alpha+1,\dots, \alpha+1}_{n-r-s}.\]

\item The vertical indices \[ \underbrace{\beta+2,\dots, \beta+2}_{t}, \qquad \underbrace{\beta+1,\dots, \beta+1}_{m-r-t}.\] 

\end{itemize} 

\end{theorem}

\begin{proof}[Proof of Theorem~\ref{T:GenLieReprKron}] Let  $L_{\hat{x},\hat{a}}$ be a generic pencil. By Theorem~\ref{T:GenericBundle} the bundle $\mathcal{B}\left( L_{x,a} \right)$ belongs to ``a lower strata'' than the bundle $\mathcal{B}\left( L_{\hat{x},\hat{a}} \right)$. More precisely, \[ \mathcal{B}\left( L_{x,a} \right) \subset \overline{\mathcal{B}}\left( L_{\hat{x},\hat{a}} \right).\]  By \eqref{Eq:RankCond} the pencil $L_{x,a}$ has the maximal possible rank: \[ r =  \operatorname{rk} L_{x,a} = \operatorname{rk} L_{\hat{x},\hat{a}}.\] By Theorem~\ref{T:GenKCFPenc} (or by Theorems~\ref{T:KCFStarts} 
 and \ref{T:BundleStrat})  the bundle $\mathcal{B}\left( L_{x,a} \right)$ is the ``maximal by inclusion'' bundle with rank $r$, i.e.
\[ \begin{cases}  \mathcal{B}\left( L_{x,a} \right) \subset \overline{\mathcal{B}}\left( L_{\hat{x},\hat{a}} \right), \\   \operatorname{rk} L_{x,a} = \operatorname{rk} L_{\hat{x},\hat{a}} \end{cases} \qquad \Rightarrow \qquad \mathcal{B}\left( L_{x,a} \right) = \mathcal{B}\left( L_{\hat{x},\hat{a}} \right) .\] We proved that  the pair $(x, a)$ is generic. Hence, the JK invaraints of $\rho$ are equal to the algebraic type of $ L_{x,a}$. Theorem~\ref{T:GenLieReprKron} is proved. \end{proof}

\begin{remark} If $n = m =r$, then the generic pencils consist of Jordan $2\times 2$ blocks with different eigenvalues and we have a similar statement to  Theorem~\ref{T:GenLieReprKron} for them. \end{remark}

\begin{example} Consider sum $\rho$ of $m > n$ standard representations of $\operatorname{sl}(n)$. As it is shown in \cite{Kozlov23} for the pair of matrices \[X = \left(\begin{matrix} I_n & 0 \end{matrix} \right), \qquad A = \left(\begin{matrix} 0 & I_n \end{matrix} \right) \] the KCF are as in Table~\ref{table:1}. The maximal possible rank for the representation $\rho$ can be easily found as \[ \codimO = \dim \mathfrak{g} - \dimSt.\] Since the pencil $L_{X+\lambda A}$ has maximal possible rank for the representation $\rho$  and it is  a generic pencil, we can apply Theorem~\ref{T:GenLieReprKron}. We get that: 
\begin{itemize}
    \item the pair $(X, A)$ is generic,  
    \item the JK invariants of $\rho$ are as in Table~\ref{table:1}.
\end{itemize} 

\end{example}

\begin{remark} \label{R:OtherRest} In practice, when calculation JK invariants of pencils we may know various restriction on them. For instance, the sum of the sizes of all Jordan blocks for a generic pencil $L$ is equal to the degree of the fundamental semi-invariant (see e.g. \cite{BolsIzosKozl19}). Then one can try to find a generic KCF that satisfy these restrictions using Theorem~\ref{T:KCFStarts} and generalize  Theorem~\ref{T:GenLieReprKron} (or prove some similar fact) for them. \end{remark}

\begin{remark} Stratification of skew-symmetric matrix pencil orbits and bundles are described in \cite{Dmytryshyn14}. Facts from there can be used for JK invariants of Lie algebras. \end{remark}

\section{Jordan--Kronecker invariants for semi-direct sums with a commutative ideal}

In this Section we prove Theorem~\ref{T:JKInvSumDualRepr} that related JK invariants of a semi-direct sum $\mathfrak{g}\ltimes_{\rho} V$ and the dual representation $\rho^*$. But first we describe the form of the matrix pencils $\mathcal{P}$ for semi-direct sums in Lemma~\ref{L:RelJKReprLie} and prove a statement about JK invariant of such matrices in Lemma~\ref{L:CalcJKCoreMantle}. We finish with some examples of applications of Theorem~\ref{T:JKInvSumDualRepr} in Section~\ref{S:ExDual}.

\subsection{Lie--Poisson bracket for semi-direct sums}

Let $\mathfrak{g}$ be Lie algebra with structural constants $c^k_{ij}$. For any $x \in \mathfrak{g}^*$ denote \[ \mathcal{A}^{\mathfrak{g}}_x = \left( \sum_k c_{ij}^k x_k \right). \] In other words, $\mathcal{A}^{\mathfrak{g}}_x $ is the matrix of the Lie--Poisson bracket on $\mathfrak{g}^*$ at the point $x$. The corresponding bilinear form on $T_{x}^* \mathfrak{g}^* \simeq \mathfrak{g}$ is given by \begin{equation} \label{Eq:LiePoissMatRel} \langle x, [\xi, \eta] \rangle = \mathcal{A}^{\mathfrak{g}}_x \left(\xi, \eta\right), \qquad \forall \xi, \eta \in \mathfrak{g}.\end{equation}

Let  $\rho : \mathfrak{g} \to \operatorname{gl} (V)$ be a linear representation of a complex finite-dimensional Lie algebra $\mathfrak{g}$ on a finite-dimensional vector space $V$. We form a new Lie algebra $\mathfrak{q}$ as the semi-direct product  $\mathfrak{q} = \mathfrak{g} \ltimes_{\rho} V$, where $V$ is an abelian ideal. For any two elements $(\xi_i, v_i)  \in   \mathfrak{g} \ltimes_{\rho} V, i =1,2,$ the commutation relations are \[ [(\xi_1, v_1), (\xi_2, v_2)] = ([\xi_1, \xi_2], \quad \rho(\xi_1) v_2 - \rho(\xi_2) v_1 ).\] We regard \[\mathfrak{q}^* = \mathfrak{g}^* + V^*, \qquad \mathfrak{g}^* = V^0, \quad V^* = \mathfrak{g}^0,\] where $V^0$ and $\mathfrak{g}^0$ are the  annihilators of $V$ and $\mathfrak{g}$ in $\mathfrak{q}^*$ respectively. We denote elements of $\mathfrak{q}^*$  as $(x, a), x \in \mathfrak{g}^*, a \in V^*$.

The JK invariants of the dual representation $\rho^*: \mathfrak{g} \to \operatorname{gl}(V^*)$ and the JK invariants of the Lie algebra $q$ are related by the following trivial statement.

\begin{lemma} \label{L:RelJKReprLie}  Let $\mathfrak{q} =   \mathfrak{g} \ltimes_{\rho} V$.  For any $(x,a) \in \mathfrak{q}^*$ the matrix of the Lie--Poisson bracket $\mathcal{A}^{\mathfrak{q}}_{(x,a)}$ has the form \[  \mathcal{A}^{\mathfrak{q}}_{(x,a)} = \left( \begin{matrix} \mathcal{A}^g_{x} & L_a^T \\ -L_a & 0 \end{matrix} \right), \]  where 

\begin{itemize}

\item $\mathcal{A}^g_{x}$ is the the matrix of the Lie--Poisson bracket for $\mathfrak{g}$ at the point $x \in \mathfrak{g}^*$; 

\item $L_a$ is the matrix of the operator \[L_a: \mathfrak{g} \to V^*, \qquad  L_a(\xi) =- \rho^*(\xi) a.  \] 

\end{itemize}

\end{lemma} 
 
 In other words, if we denote by $R^{\rho^*}_a$ the operator given by \eqref{Eq:LInOperRepr} for the dual representation $\rho^*$, then \[  L_a = -R^{\rho^*}_a.\]
 
\begin{proof}[Proof of Lemma~\ref{L:RelJKReprLie}]  We just use the formular for $\mathcal{A}^{\mathfrak{q}}_{(x,a)}$: \[ \langle (x,a), [\xi, \eta] \rangle = \mathcal{A}^{\mathfrak{q}}_{(x,a)} \left(\xi, \eta\right), \qquad \forall \xi, \eta \in \mathfrak{q}.\]  For instance, if $\xi \in \mathfrak{g}$ and $\eta = v \in V$, then \[\mathcal{A}^{\mathfrak{q}}_{(x,a)} \left(\xi, v\right) = \langle a, \rho(\xi) v\rangle = \langle -\rho^*(\xi) a, v\rangle.\]  Thus, the top-right block of $\mathcal{A}^{\mathfrak{q}}_{(x,a)}$ is $L_a^T$. Note that we have to transpose the matrix, since the top-right block is a $\dim\mathfrak{g} \times \dim V$ block and $L_a$ is a $\dim V \times \dim \mathfrak{g}$ matrix.  The other blocks of $\mathcal{A}^{\mathfrak{q}}_{(x,a)}$  are found in a similar way.  Lemma~\ref{L:RelJKReprLie}  is proved. \end{proof} 

\subsection{JK invariants for block upper-triangular pencils} \label{S:CalcJKCoreMantle}

We need the following fact. 

\begin{lemma} \label{L:CalcJKCoreMantle} Let $\mathcal{P} = \left\{ A + \lambda B\right\}$ be a pencil of skew-symmetric bilinear forms on a vector space $V$ such that their matrices have the form \begin{equation} \label{Eq:PencGoodKM} A = \left( \begin{matrix} A_{xx} & A_{xs} & A_{xy} \\ -A_{xs}^T & A_{ss} & 0 \\ -A_{xy}^T & 0 & 0\end{matrix} \right), \qquad B = \left( \begin{matrix} B_{xx} & B_{xs} & B_{xy} \\ -B_{xs}^T & B_{ss} & 0 \\ -B_{xy}^T & 0 & 0\end{matrix} \right).\end{equation} Consider two pencils \[ \mathcal{P}_K = \left( \begin{matrix} 0 & A_{xy} + \lambda B_{xy} \\ -(A_{xy} + \lambda B_{xy})^T & 0 \end{matrix}\right), \qquad \mathcal{P}_J = \left( \begin{matrix} A_{ss} + \lambda B_{ss} \end{matrix}\right).\] Assume that they satisfy the following $2$ conditions:

\begin{enumerate}

\item $\operatorname{Ker}(A_{xy} + \lambda B_{xy})^T = 0$ and $\operatorname{rk} (A_{xy} + \lambda B_{xy}) = \operatorname{const}$ for all $\lambda \in \mathbb{C} \cup \left\{ \infty \right\}$.

\item The matrix $A_{ss} + \lambda_0 B_{ss}$ is nondegenerate for some $\lambda_0\in \mathbb{C} \cup \left\{ \infty \right\}$.

\end{enumerate}

Then the JK decomposition of $\mathcal{P}$ is the union of the JK decompositions of $\mathcal{P}_K$ and $\mathcal{P}_J$. 

\end{lemma}

\begin{remark} In terms of JK invariants conditions from Lemma~\ref{L:CalcJKCoreMantle}  can be formulated as follows:

\begin{enumerate}

\item The pencil $A_{xy} + \lambda B_{xy}$ consists only of horizontal Kronecker blocks.

\item The Poisson pencil $P_{J}$ is a Jordan-pencil (i.e. its JK decomposition consists of Jordan blocks).

\end{enumerate}

\end{remark}

\begin{remark} Simply speaking in Lemma~\ref{L:CalcJKCoreMantle} we can calculate JK invariants of $\mathcal{P}$ as if the blocks \[A_{xx}= B_{xx} =0, \qquad \text{ and } \qquad A_{xs} = B_{xs} = 0.\]  Lemma~\ref{L:CalcJKCoreMantle}  also formally holds in the following $2$ cases:

\begin{itemize}

\item There are no $s$-coordinates, i.e. \[ \mathcal{P} = \left( \begin{matrix} A_{xx} + \lambda B_{xx} & A_{xy} + \lambda B_{xy} \\ -(A_{xy} + \lambda B_{xy})^T & 0 \end{matrix}\right)\] 

\item There are no $x$ and $y$-coordinates, i.e. \[\mathcal{P} = \left( \begin{matrix} A_{ss} + \lambda B_{ss} \end{matrix}\right).\]

\end{itemize}

\end{remark}

\begin{proof}[Proof of Lemma~\ref{L:CalcJKCoreMantle}] The proof is in several steps.

\begin{enumerate}

\item First, note that $\mathcal{P}_J$ is a Jordan pencil and hence $\operatorname{Ker} (A_{ss} + \lambda B_{ss}) = 0$ for almost all $\lambda \in  \mathbb{C} \cup \left\{ \infty \right\}$. 

\item The pencil $\mathcal{P}_K$ is Kronecker, since \[\operatorname{rk}\left( \begin{matrix} 0 & A_{xy} + \lambda B_{xy} \\ -(A_{xy} + \lambda B_{xy})^T & 0 \end{matrix}\right) = 2 \operatorname{rk} (A_{xy} + \lambda B_{xy}) = \operatorname{const}.\]

\item Recall that the Kronecker blocks for a pencil $\left\{ A + \lambda B\right\}$ are determined by the polynomical solutions of \[ \left( A + \lambda B\right) v(\lambda) =0.\] This follows from Proposition~\ref{L:ChainsRestrCor}. The polynomial solutions for the pencil $\mathcal{P}$ and $\mathcal{P}_K$ coincide, since \[ \Ker  (A_{xy} + \lambda B_{xy})^T = 0, \qquad \operatorname{Ker} (A_{ss} + \lambda B_{ss}) = 0 \] for almost all $\lambda \in  \mathbb{C} \cup \left\{ \infty \right\}$. Thus, the Kronecker sizes for $\mathcal{P}$ and $\mathcal{P}_K$ coincide. Since $\mathcal{P}_K$ is Kronecker, the JK invariants of $\mathcal{P}_K$ is the ``Kronecker part'' of the JK invariants of $\mathcal{P}$. 

\item The rows and columns of the pencil \eqref{Eq:PencGoodKM} are partitioned into $3$ groups. Consider the corresponding decompostion of $V$: \[ V = U_x \oplus U_s \oplus U_y.\] Let us prove that $U_y$ is the core $K$ and $U_s \oplus U_y$ is the mantle $M$ of the pencil $\mathcal{P}$. Since $M = K^{\perp}$, it suffices to prove $K = U_y$.

\begin{itemize}

\item On one hand,  $K \subset  U_y$. Indeed, on the previous step we proved that \[ \operatorname{Ker} (A + \lambda B) =  \Ker  (A_{xy} + \lambda B_{xy}) \subset U_y\]  for generic $\lambda \in \mathbb{C} \cup \left\{ \infty \right\}$. Thus, the cores of $\mathcal{P}$ and $\mathcal{P}_K$ is the same subspace $K \subset U_y$. 

\item On the other hand, since $\mathcal{P}_K$ is Kronecker, $K^{\perp_{\mathcal{P}_K}} = K \supset  U_y$. Here $K^{\perp_{\mathcal{P}_K}}$ is the subspace that is skew-orthogonal to $K$ in $U_x \oplus U_y$ w.r.t. any regular form in the pencil $\mathcal{P}_K$.

\end{itemize}

 We proved that \[K = U_y \qquad \text{ and } \qquad M = U_s \oplus U_y.\]

\item Since  $K=U_y$ and $M = U_s \oplus U_y$, the Jordan blocks for $\mathcal{P}$ and for $\mathcal{P}_J = \left( \begin{matrix} A_{ss} + \lambda B_{ss} \end{matrix}\right)$ coincide. Since $\mathcal{P}_J$ is a Jordan pencil, the JK invariants of $\mathcal{P}_J$ is the ``Jordan part'' of the JK invariants of $\mathcal{P}$, as required. 

\end{enumerate}

Lemma~\ref{L:CalcJKCoreMantle} is proved. \end{proof}

\subsection{Proof of Theorem~\ref{T:JKInvSumDualRepr}} \label{SubS:ProofThDual}

Consider a generic pair $(x_1, a_1), (x_2, a_2) \in \mathfrak{q}^*$. By Lemma~\ref{L:RelJKReprLie}  the pencil $\mathcal{P} = \mathcal{A}_{(x_1,a_1)} + \mathcal{A}_{(x_2, a_2)} $ has the form \[ \mathcal{P} =  \left( \begin{matrix} \mathcal{A}^g_{x_1} & L_{a_1}^T \\ -L_{a_1} & 0 \end{matrix} \right) + \lambda \left( \begin{matrix} \mathcal{A}^g_{x_2} & L_{a_2}^T \\ -L_{a_2} & 0 \end{matrix} \right). \]  By assumption the pencil $R^{\rho^*}_{a_1} + \lambda R^{\rho^*}_{a_2}$ for the dual representation $\rho^*$ has no horizontal Kronecker blocks. Thus,  the pencil \begin{equation} \label{Eq:TransPenc} L = L_{a_1}^T  + \lambda L_{a_2}^T = - \left( R^{\rho^*}_{a_1} + \lambda R^{\rho^*}_{a_2}\right)^T \end{equation} has no vertical Kronecker blocks. Thus, by rearranging the bases of $\mathfrak{g}$ and $V$ we can take it to the form \[L_{a_1}^T  + \lambda L_{a_2}^T = \left( \begin{matrix} 0 & M_{\text{hor.Kron}} \\ M_{\text{Jord}} & 0  \end{matrix} \right),\] where  $M_{\text{hor.Kron}}$ consists only of horizontal Kroncker blocks and  $M_{\text{Jord}}$ consists only of Jordan blocks. The matrix $\mathcal{P}$ takes the form \[ \mathcal{P} =  \left( \begin{matrix} * & * & 0 & M_{\text{hor.Kron}} \\ * & * & M_{\text{Jord}} & 0  \\ 0 & -M^T_{\text{Jord}} & 0 & 0  \\ -M^T_{\text{hor.Kron}}  & 0 & 0 & 0 \end{matrix} \right). \] Applying Lemma~\ref{L:CalcJKCoreMantle}, we get the following: 

\begin{itemize} 

\item The Kronecker indices for $\mathcal{P}$ coincide with the sizes for the horizontal Kronecker blocks for $M_{\text{hor.Kron}}$. By \eqref{Eq:TransPenc} those are precisely vertical indices for the dual representation $\rho^*$.

\item The Jordan blocks for $\mathcal{P}$ are the same as in the matrix \[ Y = \left( \begin{matrix}  * & M_{\text{Jord}}  \\  -M^T_{\text{Jord}} & 0 \end{matrix} \right). \]  It is easy to see that for each Jordan block with eigenvalue $\lambda_0$ in   $R^{\rho^*}_{a_1} + \lambda R^{\rho^*}_{a_2}$ we get a "skew-symmetric pair" of Jordan blocks with eigenvalue $-\lambda_0$ in the matrices $M_{\text{Jord}}$ and $-M_{\text{Jord}}^T$. The Pffafian of $Y$ is the characterstic polynomial for $M_{\text{Jord}}$. Thus, by Proposition~\ref{cor2} each 
Jordan tuple $J_{\lambda_i} \left(n_{i1}, \dots, n_{i s_i}\right)$ for the dual representation corresponds to some Jordan tuple  $J^{\operatorname{skew}}_{\lambda_i} \left(2m_{i1}, \dots, 2m_{i t_i}\right)$ in $Y$, where  \[m_{i1} + \dots + m_{it_i} = n_{i1} + \dots + n_{is_i}.\] The sum of sizes of Jordan blocks coincide, since they are determined by the multiplicities of the corresponding roots in the characteristic polynomials. The sizes of Jordan blocks can change because of the upper-left block $*$ in $Y$.
\end{itemize}

Theorem~\ref{T:JKInvSumDualRepr} is proved.

\subsection{Examples} \label{S:ExDual}

In these section we consider sums of $m$ standard representation of matrix Lie algebras
$\mathfrak{g} \subset \mathfrak{gl}(n)$. Simply speaking, we can identify the vector space $V$ with $n \times m$ matrices $\operatorname{Mat}_{n \times m}$. Then  corresponding linear pencils are given by the right multiplication of matrices: \begin{align*} R_X: \operatorname{Mat}_{n\times n} & \to \operatorname{Mat}_{n\times m} \\ Y & \to YX.\end{align*}

\begin{itemize}
    \item First, in Section~\ref{SubS:JKDual} we prove that in these cases the JK invariants for the sum of standard representions and their dual coincide. 
    
    \item Then, in Section~\ref{S:SubSStandRep} we consider JK invariants for sums of standard representation of $\operatorname{gl}(n)$, $\operatorname{sl}(n)$, $\operatorname{so}(n)$ and $\operatorname{sp}(n)$. 
\end{itemize}

\subsubsection{Jordan--Kronecker invariants for dual standard representations} \label{SubS:JKDual}

For any representation $\rho:\operatorname{g} \to \operatorname{gl}(V)$ we denote the sum of $m$ representations as \[\rho^{\oplus_m} : \operatorname{g} \to \operatorname{gl}\left(V^{\oplus_m}\right).\] First, we prove the following simple statement.

\begin{lemma} \label{L:SelfDualJK} Let $\mathfrak{g} \subset \mathfrak{gl}(n)$ be a matrix Lie algebra and $\rho: \mathfrak{g}  \to \operatorname{gl}(\mathbb{C}^n) $ be its standard representation. If for any $X \in \mathfrak{g}$ we also have $-X^T \in \mathfrak{g}$, then the sums of $m$ standard representations  $\rho^{\oplus_m}$ and its dual $\left(\rho^{\oplus_m}\right)^*$ have the same Jordan--Kronecker invariants of representations. \end{lemma}

Note that in Lemma~\ref{L:SelfDualJK} the representation $\rho$ doesn't have to be self-dual. 

\begin{proof}[Proof of Lemma~\ref{L:SelfDualJK}] Denote $\rho_m = \rho^{\oplus_m}$, for short. Identify $\operatorname{Mat}^*_{n\times m}$ with $\operatorname{Mat}_{n \times m}$ by  \[ \langle A, B \rangle = \operatorname{tr} \left(A^T B \right).\] The dual map $\rho_m^*$ is given by the formula \[ \rho_m^*(Y) X = -Y^T X. \] Thus, the corresponding pencil is given by the following linear map: \begin{align*} \hat{R}_X: \operatorname{Mat}_{n\times n} & \to \operatorname{Mat}_{n\times m} \\ Y & \to -Y^TX.\end{align*} The pencils $R_{X_1 + \lambda X_2}$ and $\hat{R}_{X_1 + \lambda X_2}$ are strictly equivalent, since there is the following commutative diagram:\[\begin{tikzcd}
\mathfrak{g} \arrow[r, "R_X"] \arrow[d, "\phi"]
& \operatorname{Mat}_{n \times m} \arrow[d, "\psi"] \\
\mathfrak{g} \arrow[r, "\hat{R}_X"]
& \operatorname{Mat}_{n \times m} 
\end{tikzcd}\] where \[ \phi(X) = - X^T, \qquad \psi(Y) = Y.\] Note that $\phi$ is a linear automophism of $\mathfrak{g}$ and may not preserve the structure of Lie algebra. Strictly equivalent pencils $R_{X_1 + \lambda X_2}$ and $\hat{R}_{X_1 + \lambda X_2}$ have the same JK invariants. Lemma~\ref{L:SelfDualJK} is proved. \end{proof}

\subsubsection{Standard representation of semi-simple Lie algebras} \label{S:SubSStandRep}

JK invariants for sums of $m$ standard representation of standard Lie algebras $\operatorname{gl}(n)$, $\operatorname{sl}(n)$, $\operatorname{so}(n)$ and $\operatorname{sp}(n)$ were calculated in 
\cite{BolsIzosKozl19} (see also \cite{Kozlov23}). By Lemma~\ref{L:SelfDualJK}  the JK invariants for the corresponding dual representations are the same. JK invariants for the corresponding semi-direct sums $\mathfrak{q} = \mathfrak{g} \ltimes \left(\mathbb{C}^n\right)^m$ were found in \cite{Vor1} and \cite{Vor2} except for some cases, when  $m < n$ and the Lie algebra is either $\operatorname{gl}(n)$ or $\operatorname{sl}(n)$. In these special cases the JK invariants are known in the following cases:  

\begin{itemize}
\item $\mathfrak{g} = \operatorname{gl}(n)$ or $\operatorname{sl}(n)$, and $n$ is a multiple of $m$ (see \cite{Vor2});

\item $\mathfrak{g} = \operatorname{sl}(n)$, $m < n$ and $n = \pm 1 \pmod{m} $ (see \cite{Kozlov24SL}).

\end{itemize}

The information is collected in Tables~\ref{table:1} and \ref{table:2}.  

\begin{enumerate}
    \item  If $m > n$, then there are only Kronecker blocks and we can apply Theorem~\ref{T:JKInvSumDualRepr}. It is easy to see that the right columns of Tables~\ref{table:1} and \ref{table:2} "coincide".  
    
    \item If $m = n$, then there are still no horizontal Kronecker blocks. Thus, the Kronecker indices in the "middle columns"  of Tables~\ref{table:1} and \ref{table:2} are the same. But Jordan blocks sometimes may "group together" for the corresponding Lie algebra (as in the case of $\operatorname{gl}(n)$).

    \item If $m < n$, then we can not apply Theorem~\ref{T:JKInvSumDualRepr}. There is no clear connection between the left columns of Tables~\ref{table:1} and \ref{table:2}.

\end{enumerate}

\newpage

\begin{table}[h!]
\centering
\renewcommand{\arraystretch}{1.5}
\begin{tabular}{|p{15mm}| p{45mm} | p{45mm} | p{45mm}|} 
\hline   & $m < n$  \newline  $m = q (n-m) + r$  & $m = n$ & $m> n$ \newline $n = q (m-n) + r$ \\ [0.5ex]  \hline\hline  \rule{0pt}{5ex}
$\operatorname{gl}(n)$ 
 & 
 Horizontal indices: \newline $ \underbrace{q+1, \dots, q+1}_{n(n-m-r)},$  \newline\newline $\underbrace{q+2, \dots, q+2}_{nr}$ 
 & Jordan tuples: \newline $J_{\lambda_j}(\underbrace{1, \dots, 1}_n), \,\, j=1,\dots, n$ &  
 Vertical indices: \newline $\underbrace{q+1, \dots, q+1}_{n(m-n-r)}, $ \newline \newline $\underbrace{q+2, \dots, q+2}_{nr}$ \\ [5ex]   \hline \rule{0pt}{5ex}
 $\operatorname{sl}(n)$ &  Horizontal indices: \newline 
 $\underbrace{q+1, \dots, q+1}_{n(n-m-r) -(q+1)}$ 
 \newline\newline $\underbrace{q+2, \dots, q+2}_{nr+q}$
 & 
Vertical index: \newline $\mathrm{v}_1 =n$ 
\newline
 \newline
Jordan tuples: \newline $J_{\lambda_j}(\underbrace{1, \dots, 1}_{n-1}), \,\, j=1,\dots, n$
&  Vertical indices $(r \not = 0)$: \newline $\underbrace{q+1, \dots, q+1}_{n(m-n-r) + (q+2) }$  \newline\newline $\underbrace{q+2, \dots, q+2}_{nr - (q+1)}$ \newline  \newline Vertical indices $(r = 0)$: $\underbrace{q, \dots, q}_{q+1}, \,\, \underbrace{q+1, \dots, q+1}_{n(m-n)-q}$  \\ [5ex]   \hline \rule{0pt}{5ex}
 $\operatorname{so}(n)$ \newline $(\varepsilon = -1)$ \newline \newline and \newline \newline $\operatorname{sp}(n)$ \newline $(\varepsilon = 1)$
  & Horizontal indices:  \newline $\underbrace{2q+1, \dots, 2q+1}_{\frac{(n-m-r)(n-m-r + \varepsilon)}{2}}$  \newline \newline $\underbrace{2q+2, \dots, 2q+2}_{(n-m-r)r}$  \newline \newline $\underbrace{2q+3, \dots, 2q+3}_{ \frac{r(r + \varepsilon)}{2}}$ \newline  \newline Vertical indices:  $\underbrace{2, \dots, 2}_{\frac{m(m - \varepsilon)}{2}}$ \newline
  & 
 Vertical indices, $\operatorname{so}(n)$:   \newline $\underbrace{1, \dots, 1}_{n}, \qquad \underbrace{2, \dots, 2}_{\frac{n(n - 1)}{2} }$  \newline   \newline  Vertical indices, $\operatorname{sp}(n)$: \newline $ \underbrace{2, \dots, 2}_{\frac{n(n - 1 )}{2} }$ \newline  \newline Jordan tuples, $\operatorname{sp}(n)$:   \newline $J_{\lambda_j}(1), \quad j=1,\dots, n$
  & Vertical indices: \newline $\underbrace{1, \dots, 1}_{(m-n -\varepsilon)n}, \quad \underbrace{2, \dots, 2}_{\frac{n(n + \varepsilon)}{2} }$  \\ [5ex] 
 \hline
\end{tabular}
\caption{JK invariants for sums of $m$ standard representations of simple Lie algebras}
\label{table:1}
\end{table}

\newpage

\begin{table}[h!]
\centering
\renewcommand{\arraystretch}{1.5}
\begin{tabular}{|p{15mm}| p{65mm} | p{35mm} | p{50mm}|} 
\hline $\mathfrak{g}$  & $m < n$  \newline  $n = l m + d$  & $m = n$ & $m> n$ \newline $n = q (m-n) + r$ \\ [0.5ex]  \hline\hline  \rule{0pt}{5ex}
$\operatorname{gl}(n)$ 
 & 
Jordan tuples $(d = 0)$: \newline $J^{\operatorname{skew}}_{\lambda_j}(\underbrace{2, \dots, 2}_{m-2}, 4),$ \, $j = 1, \dots, \frac{m l (l+1)}{2}$ 
 & Jordan tuples: \newline $J^{\operatorname{skew}}_{\lambda_j}(\underbrace{2, \dots, 2}_{n-2}, 4)$,  $j=1,\dots, n$ &  
 Kronecker indices: \newline $\underbrace{q+1, \dots, q+1}_{n(m-n-r)}, $  \,\ $\underbrace{q+2, \dots, q+2}_{nr}$ \\ [5ex]   \hline \rule{0pt}{5ex}
 $\operatorname{sl}(n)$ &  Kronecker index $(d=0)$: \newline $\frac{ml(l+1)}{2}$, \newline \newline Jordan tuples $(d=0)$: \newline $J^{\operatorname{skew}}_{\lambda_j}(\underbrace{2, \dots, 2}_{m-1}),$  \,\ $j = 1, \dots, \frac{ml(l+1)}{2}$ \newline \newline  Kronecker indices ($d = 1$ or $m-1$):  \newline $ \underbrace{\frac{(l+1)(n+d)}{2}, \dots, \frac{(l+1)(n+d)}{2}}_{m}$ \newline 
 & 
Kronecker index: \newline $k_1 = n$ 
\newline
 \newline
Jordan tuples: \newline $J^{\operatorname{skew}}_{\lambda_j}(\underbrace{2, \dots, 2}_{n-1})$, \newline \newline $ j=1,\dots, n$
&   Kronecker indices $(r \not = 0)$: \newline $\underbrace{q+1, \dots, q+1}_{n(m-n-r) + q+2 }$,   $\underbrace{q+2, \dots, q+2}_{nr -(q+1)}$ \newline  \newline  Kronecker indices $(r = 0)$: $\underbrace{q, \dots, q}_{q+1},$ \,\, $\underbrace{q+1, \dots, q+1}_{n(m-n)-q}$  \\ [5ex]   \hline \rule{0pt}{5ex}
 $\operatorname{so}(n)$ \newline $\varepsilon = -1$ 
  &  Kronecker indices  $(n+m = 2a + 1)$: \newline $\underbrace{2, \dots, 2}_{\frac{m(m - \varepsilon)}{2}}$, \,\,  $\underbrace{2m+2, \dots,  n+m-1}_{\text{even numbers}}$ \newline  \newline Kronecker indices,  $(n+m = 2a)$: \newline $\underbrace{2, \dots, 2}_{\frac{m(m - \varepsilon)}{2}}$, \, $\underbrace{2m+2, \dots, n+m-2}_{\text{even numbers}}, \, \frac{n+m}{2}$
  & 
 Kronecker indices:   \newline $\underbrace{1, \dots, 1}_{n},  \,\,\underbrace{2, \dots, 2}_{\frac{n(n+ \varepsilon)}{2} }$ 
  &  Kronecker indices: \newline $\underbrace{1, \dots, 1}_{(m-n -\varepsilon)n}$, \,\, $\underbrace{2, \dots, 2}_{\frac{n(n + \varepsilon)}{2} }$  \\ [5ex] 
 \hline
$\operatorname{sp}(n)$ \newline $\varepsilon = 1$
  &  Kronecker indices $(n+m = 2a+1)$: \newline $\underbrace{2, \dots, 2}_{\frac{m(m - \varepsilon)}{2}}$, \,\, $\underbrace{2m+1, \dots, n+m}_{\text{odd numbers}}$ \newline \newline Kronecker indices $(n+m = 2a)$: \newline $\underbrace{2, \dots, 2}_{\frac{m(m - \varepsilon)}{2}}$,  \,\, $\underbrace{2m+2, \dots, n+m}_{\text{even numbers}}$ \newline \newline Jordan tuples $(n+m = 2a)$: \newline $J^{\operatorname{skew}}_{\lambda_j}(2), \quad j=1,\dots, m$
  & 
  Kronecker indices: \newline $ \underbrace{2, \dots, 2}_{\frac{n(n- 1)}{2} }$ \newline  \newline Jordan tuples:   \newline $J^{\operatorname{skew}}_{\lambda_j}(2)$, \newline $j=1,\dots, n$
  &  Kronecker indices: \newline $\underbrace{1, \dots, 1}_{(m-n -\varepsilon)n}, \quad \underbrace{2, \dots, 2}_{\frac{n(n + \varepsilon)}{2} }$  \\ [5ex] 
 \hline
\end{tabular}
\caption{JK invariants for semi-direct sums $\mathfrak{q} = \mathfrak{g} \ltimes \left(\mathbb{C}^n\right)^m$}
\label{table:2}
\end{table}

\end{document}